\documentclass[english]{amsart}%
\usepackage{amsfonts}
\usepackage{amsmath}
\usepackage{amssymb}
\usepackage{amsthm}
\usepackage{amscd}
\usepackage{babel}
\usepackage[latin1]{inputenc}
\usepackage{graphicx}%
\providecommand{\U}[1]{\protect\rule{.1in}{.1in}}
\newtheorem{teor}{Theorem}
\newtheorem{cor}{Corollary}
\newtheorem{prop}{Proposition}
\newtheorem{con}{Conjecture}
\newtheorem{lem}{Lemma}

\theoremstyle{definition}

\newtheorem*{rem}{Remark}

\renewcommand{\subjclassname}{AMS \textup{2010} Mathematics Subject
Classification\ }
\email{bayon@uniovi.es}
\email{grau@uniovi.es}
\email{oller@unizar.es}
\email{mruiz@uniovi.es}
\email{pedrosr@uniovi.es}
\begin{document}
\author{L. Bayón}
\address{Departamento de Matemáticas, Universidad de Oviedo\\
Avda. Calvo Sotelo s/n, 33007 Oviedo, Spain}
\author{J. Grau}
\address{Departamento de Matemáticas, Universidad de Oviedo\\
Avda. Calvo Sotelo s/n, 33007 Oviedo, Spain}
\author{A. M. Oller-Marcén}
\address{Centro Universitario de la Defensa de Zaragoza\\
Ctra. Huesca s/n, 50090 Zaragoza, Spain}
\author{M. Ruiz}
\address{Departamento de Matemáticas, Universidad de Oviedo\\
Avda. Calvo Sotelo s/n, 33007 Oviedo, Spain}
\author{P.M. Suárez}
\address{Departamento de Matemáticas, Universidad de Oviedo\\
Avda. Calvo Sotelo s/n, 33007 Oviedo, Spain}
\title{A variant of the Secretary Problem: the Best or the Worst}

\begin{abstract}

We consider a variant of the secretary problem in which
the candidates state their expected salary at the interview, which we assume
is in accordance with their qualifications. The goal is for the employer to
hire the best or the worst (cheapest), indifferent between the two cases. We focus on the complete information variant as well as on the
cases when the number of applicants is a random variable with a uniform
distribution $U[1,n]$ or with a Poisson distribution of parameter $\lambda$.
Moreover, we also study two variants of the original problem in which we
consider payoffs depending on the number of conducted interviews.

\end{abstract}
\maketitle
\keywords{Keywords: Secretary problem, Combinatorial Optimization}

\subjclassname{60G40, 62L15}

\section{Preliminaries}

The optimal stopping problem, known as the classical \emph{secretary problem},
can be stated as follows: an employer is willing to hire the best secretary
out of $n$ rankable candidates for a position. The candidates are interviewed
one by one in a random order. A decision about each particular candidate is to
be made immediately after the interview. Once rejected, a candidate cannot be
recalled. During the interview, the administrator can rank the applicant among
all applicants interviewed so far, but is unaware of the quality of yet unseen
applicants. The question is about the optimal strategy (stopping rule) to
maximize the probability of selecting the best candidate. The secretary
problem is one of many names for a famous problem of optimal stopping theory. It is
also known as the marriage problem, the sultan's dowry problem, the fussy
suitor problem and others. This problem has been extensively studied in the fields of applied probability,
statistics, and decision theory and has been considered by many authors (see
\cite{FER}, \cite{FER2} and \cite{2009} for an extensive bibliography). It can also  be posed as a decision-taking
problem in a game with the following rules:

\begin{enumerate}
\item You want to choose one object.
\item The number of objects is known.
\item The objects appear sequentially in random order.
\item The objects are rankable.
\item Each object is accepted or rejected before the next object appears.
\item The decision depends only on the relative ranks.
\item Rejected objects cannot be recalled.
\item Payoff: You only win if the best object is chosen.
\end{enumerate}

This problem has a very elegant solution. Dynkin
\cite{48} and Lindley \cite{101} independently proved that the best strategy consists in
observing roughly $n/e$ of the candidates and then choosing the first one that
is better than all those observed so far. This strategy returns the best
candidate with a probability of at least $1/e$, this being its approximate
value for large values of $n$. This well-known solution was refined by Gilbert
and Mosteller in \cite{gil}, where they show that $\left[  (n-\frac{1}%
{2})e^{-1}+\frac{1}{2}\right]  $ is a better approximation than $[n/e]$,
although the difference is never greater than 1. There are other variants of the
secretary problem that also have simple, elegant solutions. In the ``postdoc''
problem, for instance, the desire to pick the best is replaced by the desire to
pick the second best (because, according to Vanderbei \cite{posdoc}, the
``best'' will go to Harvard). For this problem, the probability of success for
an even number of applicants is exactly $\frac{n}{4(n-1)}$. This probability
tends to 1/4 as $n$ tends to infinity, illustrating the fact that it is easier
to pick the best than the second best.

If the number of candidates is unknown, the observer faces an additional
risk. If he rejects any candidate, he may then discover that it was the last one,
in which case he receives nothing at all. In \cite{sonin} the case in
which the number of candidates is a random variable with a discrete uniform
distribution $U[1,n]$ was studied. In this case, the cutoff value for large $N$ is
approximately $Ne^{-1/2}$ and the probability of success is $2e^{-2}$. This
same paper also tackles the problem assuming that the number of candidates is
Poisson distributed with parameter $\lambda$.

Another interesting variant of the problem was introduced by Bearden \cite{KK}, considering a positive payoff that increases
with the number interviews. In this case, the optimal cutoff value
is not proportional to the number of candidates and it is the square root of such number.

In the present paper, we consider a variant of the secretary problem in which
the candidates state their expected salary at the interview, which we assume
is in accordance with their qualifications. The goal is for the employer to
hire the best or the worst (cheapest), indifferent between the two cases. In
other words, it is the classical problem in which rule 8 above is replaced by:
``Payoff: You only win if the best or the worst is chosen''. The problem is
studied in Section 2 in its version with full information regarding the number
of candidates. In Section 3, we consider the problem if the number of
candidates is a random variable with a uniform distribution $U[1,n]$. In Section 4,
the number of candidates is a random variable with a Poisson distribution
of parameter $\lambda$. In Section 5, we consider payoffs other than the
classical binary payoff that are dependent on the number of interviews carried
out. In one case, we consider payoff proportional to the number of interviews
and in another, payoff proportional to the number of interviews not carried
out.

The techniques used throughout the paper are suitable adaptations of
those employed in the study of the solution to the classical secretary problem
(limit of Riemann integrals, expansions at
infinity, etc.) to obtain asymptotic formulas of the type $\alpha\cdot n+\beta+o(1)$
for the optimal cutoff value. In some proofs related to the calculation of
some limits, details have been omitted and their correctness has been tested
with the powerful symbolic computation tool Mathematica. Several of the
results are formulated in terms of $W_{0}$ and $W_{-1}$ functions, the lower real branches of the Lambert $W$ function,
respectively; see Corless et al. \cite{lam} for further details.

\section{A variation of the secretary problem: the best or the worst}

In this section, we study a variation of the classical secretary problem
considering that the success of the player resides in choosing the worst or
the best object, indifferent between the two cases. In this case, as in the secretary problem, the strategies to consider
consist in rejecting a certain number of objects (cutoff value) and then
taking the first object that is better or worse than all those previously
rejected. After rejecting $r$ objects, the probability of successfully
choosing the $k$-th inspected object, $\mathcal{P}(k)$, is that of this object
being the best or the worst and that of not having interrupted the inspections
by previously choosing another object. Given that both are independent events,
the probability of success will be the product of the probability of both
events. The probability that the $k$-th object to be seen is the best or worst
of all the objects is $2/n$ while the probability that the best and the worst
among those seen prior to the $k$-th is among the first $r$ objects (i.e. the
inspections have not ceased before seeing the $k$-th object) is
\[
\frac{\binom{r}{2}}{\binom{k-1}{2}}=\frac{r(r-1)}{(k-1)(k-2)}.
\]

Thus, provided $1<r<k\leq n$, we have that
\[
\mathcal{P}(k)=\frac{2}{n} \frac{\binom{r}{2}}{\binom{k-1}{2}}
\]
and the probability of success, rejecting the first $r$ objects, and accepting
afterwards the first one which is either better than all the previous ones or
worse than all the previous ones, is
\[
P_{n}(r):=\sum_{k = r + 1}^{n}\mathcal{P}(k)=\frac{2\binom{r}{2}}{n} \sum_{k =
r + 1}^{n} \frac{1}{\binom{k-1}{2}}.
\]

Note that for $n>r\in\{0,1\}$, it is straightforward to see that the
probability of success is%

\[
P_{n}(0)=P_{n}(1)=\frac{2}{n}.
\]

\begin{teor}
Given a positive integer $n>2$, consider the function
\[
P_{n}(r):=\frac{2\binom{r}{2}}{n}\sum_{k=r+1}^{n}\frac{1}{\binom{k-1}{2}}%
\]
defined for every integer $r\in[2,n-1]$. Let us denote by $\mathcal{M}(n)$ the
value for which the function $P_{n}$ reaches its maximum. Then,

\begin{itemize}
\item[i)] $P_{n}(r)=\displaystyle{\frac{2r(n-r)}{n(n-1)}}$.
\item[ii)] $\mathcal{M}(n)=\lfloor n/2\rfloor$.
\item[iii)] The maximum of $P_{n}$ is:
\[
\mathfrak{p}_{n}:= P_{n}(\mathcal{M}(n))=\frac{\lfloor\frac{1+n}{2}\rfloor
}{2\lfloor\frac{1+n}{2}\rfloor-1}=%
\begin{cases}
\frac{n}{2(n-1)}, & \text{if $n$ is even};\\
\frac{n+1}{2n}, & \text{if $n$ is odd}.
\end{cases}
\]

\end{itemize}
\end{teor}

\begin{proof}
\

\begin{itemize}
\item[i)] First of all, observe that
\[
P_{n}(r)=\frac{2\binom{r}{2}}{n}\sum_{k=r+1}^{n}\frac{1}{\binom{k-1}{2}}%
=\frac{r(r-1)}{n}\sum_{k=r+1}^{n}\frac{2}{(k-1)(k-2)}.
\]
Now, using telescopic sums, we have that
\begin{align*}
\sum_{k=r+1}^{n}\frac{2}{(k-1)(k-2)}  &  =2\sum_{k=r+1}^{n}\left(  \frac
{1}{k-2}-\frac{1}{k-1}\right)  =2\left(  \frac{1}{r-1}-\frac{1}{n-1}\right)
=\\
&  =\frac{2(n-r)}{(r-1)(n-1)}%
\end{align*}
and the result follows.

\item[ii)] Since $P_{n}(r)=-\frac{2}{n(n-1)}r^{2}+\frac{2}{(n-1)}r$ is the
equation of a parabola in the variable $r$, it is clear that
\[
\mathcal{M}(n)=\min\left\{  r\in[2,n-1]:P_{n}(r)\geq P_{n}(r+1)\right\}  .
\]
Now,
\[
P_{n}(r+1)-P_{n}(r)=\frac{2}{n(n-1)}(n-2r-1)
\]
so it follows that
\[
P_{n}(r+1)-P_{n}(r)\leq0\Leftrightarrow(n-2r-1)\leq0\Leftrightarrow r\geq
\frac{n-1}{2}.
\]
Consequently,
\[
\mathcal{M}(n)=\min\left\{  r\in[2,n-1]:r\geq\frac{n-1}{2}\right\}  =\lfloor
n/2\rfloor
\]
as claimed.

\item[iii)] It is enough to apply the previous result.

If $n$ is even, then $n=2N$ and
\[
P_{n}(\mathcal{M}(n))=P_{n}(N)=\frac{2N(n-N)}{n(n-1)}=\frac{2N^{2}}%
{2N(2N-1)}=\frac{N}{2N-1}%
\]
Moreover, in this case
\[
\left\lfloor \frac{1+n}{2}\right\rfloor =\left\lfloor \frac{1+2N}%
{2}\right\rfloor =N
\]
so it follows that
\[
P_{n}(\mathcal{M}(n))=\frac{N}{2N-1}=\frac{\left\lfloor \frac{1+n}%
{2}\right\rfloor }{2\left\lfloor \frac{1+n}{2}\right\rfloor -1}%
\]
as claimed.

Otherwise, if $n$ is odd, then $n=2N+1$ and
\[
P_{n}(\mathcal{M}(n))=P_{n}(N)=\frac{2N(n-N)}{n(n-1)}=\frac{2N(2N+1-N)}%
{(2N+1)2N}=\frac{N+1}{2N+1}.
\]
In this case
\[
\left\lfloor \frac{1+n}{2}\right\rfloor =\left\lfloor \frac{1+2N+1}%
{2}\right\rfloor =N+1
\]
so we also have that
\[
P_{n}(\mathcal{M}(n))=\frac{N+1}{2N+1}=\frac{\left\lfloor \frac{1+n}%
{2}\right\rfloor }{2\left\lfloor \frac{1+n}{2}\right\rfloor -1}%
\]
and the proof is complete.
\end{itemize}
\end{proof}

It can be deduced from this result that, for $n>2$, the strategy that
maximizes the probability of success consists in rejecting the first
$\lfloor\frac{n}{2}\rfloor$ (optimal cutoff value) objects and then choosing
the first to appear that is better or worse than all the preceding ones. The
probability of success following this strategy is
\[
\frac{\lfloor\frac{1+n}{2}\rfloor}{2\lfloor\frac{1+n}{2}\rfloor-1}%
\]

In the cases $n\in\{1,2\}$, it is evident that an optimal cutoff value is
$r=0$, i.e. to accept the first object that we are shown; the value of the
probability of success being 1 in both cases: $\mathfrak{p}_{1}=\mathfrak{p}%
_{2}=1$

\section{Unknown number of candidates with uniform distribution $U[1,n]$}

In this section, we analyze the case in which the number of objects is
unknown, but it is known that it can be anywhere between $1$ and $n$ with
equal probability. As we saw in Theorem 1 i), the probability of success with
$i>1$ objects while rejecting the first $r$ observed objects is $P_{i}%
(r)=\frac{2r(i-r)}{i(i-1)}$. Hence, the probability of success in this case
will be the mean probability of the $n$ equally probable possible scenarios
with probability $1/n$. That is, the probability of success when rejecting the
$r$ first objects before comparing with the previous ones and with an unknown
number of candidates that follows a uniform distribution $U[1,n]$ will be:%

\[
P(r,n):= \sum_{i=r+1}^{n} \frac{P_{i}(r)}{n}= \sum_{i=r+1}^{n} \frac
{2r(i-r)}{i(i-1)n}%
\]

We denote by $\mathfrak{m}(n)$ the value for which $P(\cdot,n)$ reaches its
maximum value in $[1,n]$ and by $\mathbf{P}(n)$, the maximum value it reaches,
i.e. $\mathbf{P}(n):=P(\mathfrak{m}(n),n)$. We shall first prove that
the probability of success in the game, $\mathbf{P}(n)$, is strictly
decreasing for $n>1$, where clearly $\mathbf{P}(1)=\mathbf{P}(2)=1$ and
$\mathfrak{m}(1)=\mathfrak{m}(2)=0$. The technique used in the proof
is an adaptation of the strategy-stealing argument: starting from any given
strategy in the game with a number of objects following a uniform distribution
$[1,n+1]$, two strategies are developed for the game with a number of objects
following a uniform distribution $[1,n]$, one or other of which necessarily
improves the initial probability of success. More specifically, we show that
if $r$ is the optimal cutoff value for the problem associated with a uniform
distribution $[1,n+1]$, then a higher probability of success is achieved in
the problem associated with a uniform distribution $[1,n]$ with either $r$ or
$r-1$ as the cutoff value.

\begin{lem}
\label{L:in} Consider the function
\[
P(r,n):=\sum_{i=r+1}^{n}\frac{2(i-r)r}{i(i-1)n}%
\]
defined for every pair of integers $(r,n)$ with $0<r<n$. Then, one of the
following inequalities holds:
\begin{align*}
P(r,n+1)  &  <P(r-1,n),\\
P(r,n+1)  &  <P(r,n).
\end{align*}
\end{lem}

\begin{proof}
Let us denote $H(n,r):=\displaystyle\sum_{i=r}^{n-1}\frac{1}{i}$. Then we have
that
\begin{align*}
P(r,n)  &  =\sum_{i=r+1}^{n}\frac{2(i-r)r}{i(i-1)n}=\frac{2r}{n}\sum
_{i=r+1}^{n}\left(  \frac{r}{i}+\frac{1}{i-1}+\frac{r}{i-1}\right)  =\\
&  =\frac{2r}{n}\left[  \left(  \frac{r}{n}-1\right)  +\sum_{i=r+1}^{n}%
\frac{1}{i-1}\right]  =\frac{2r}{n}\left[  \left(  \frac{r}{n}-1\right)
+H(n,r)\right]  .
\end{align*}

In the same way we find that
\begin{align*}
P(r,n+1)  &  =\frac{2r}{n+1}\left[  \left(  \frac{r}{n+1}-1+\frac{1}%
{n}\right)  +H(n,r)\right]  ,\\
P(r-1,n)  &  =\frac{2(r-1)}{n}\left[  \left(  \frac{r-1}{n}-1+\frac{1}%
{r-1}\right)  +H(n,r)\right]
\end{align*}

As a consequence, it follows that
\begin{align*}
P(r,n+1)-P(r-1,n)  &  =\frac{2(1+n-r)}{n(n+1)}\left[  -\frac{(1+2n)(1+n-r)}%
{n(n+1)}+H(n,r)\right]  ,\\
P(r,n+1)-P(r,n)  &  =\frac{2r}{n(n+1)}\left[  \frac{2n(1+n-r)-r}%
{n(n+1)}-H(n,r)\right]
\end{align*}

Thus, if we assume that both
\[
P(r,n+1)-P(r-1,n)\geq0\ \text{and}\ P(r,n+1)-P(r,n)\geq0,
\]
it follows that
\begin{align*}
A  &  :=H(n,r)n(n+1)-(1+2n)(1+n-r)\geq0,\\
B  &  :=-\left(  H(n,r)n\,\left(  n+1\right)  \right)  +2n\left(
1+n-r\right)  -r\geq0.
\end{align*}

But in this case $0\leq A+B=-1-n<0$, a contradiction, and hence the result.
\end{proof}

\begin{cor}
\label{C:dec} With the previous notation, $\mathbf{P}(n+1)<\mathbf{P}(n)$ for
every $n>1$.
\end{cor}

\begin{proof}
If $n>1$,
\begin{align*}
\mathbf{P}(n+1)  &  =P(\mathfrak{m}(n+1),n+1)<\max\{P(\mathfrak{m}%
(n+1),n),P(\mathfrak{m}(n+1)-1,n)\}\leq\\
&  \leq P(\mathfrak{m}(n),n)=\mathbf{P}(n),
\end{align*}
where the first inequality follows from Lemma \ref{L:in} and the second holds
by the definition of $\mathfrak{m}(n)$.
\end{proof}

\begin{lem}
\label{L:max} The function $g(x)=-2x\log x-2x(1-x)$ reaches its absolute
maximum in the interval $[0,1]$ at the point
\[
\vartheta:=-\frac{1}{2}W(-\frac{2}{e^{2}})=0.20318786...,
\]
where $W$ denotes Lambert $W$-function. Moreover, the value of this maximum
is:
\[
g(\vartheta)=2(\vartheta-\vartheta^{2})=0.32380511...
\]
\end{lem}
\begin{proof}
This is an elementary Calculus exercise. The derivative $g^{\prime
}(x)=-4+4x-2\log x=0$ in $[0,1]$ if and only if $x=1$ or $x=-\frac{1}%
{2}W\left(  -\frac{2}{e^{2}}\right)  $. It is easily checked that $x=1$ is a
minimum and that $x=-\frac{1}{2}W\left(  -\frac{2}{e^{2}}\right)  $ is a
maximum and some easy computations conclude the proof.
\end{proof}

\begin{teor}
With all the previous notation, the following hold:
\begin{itemize}
\item[i)] For every positive integer $n$,
\[
1=\mathbf{P}(1)= \mathbf{P}(2)>\mathbf{P}(3)>\cdots>\mathbf{P}(n)>\mathbf{P}%
(n+1)>\cdots>2(\vartheta-\vartheta^{2}).
\]
\item[ii)] $\displaystyle \lim_{n\rightarrow\infty} \frac{\mathfrak{m}(n)}%
{n}=\vartheta$; i.e., $\mathfrak{m}(n) \sim\vartheta n.$
\item[iii)] $\displaystyle \lim_{n\rightarrow\infty} P(\lfloor\vartheta n
\rfloor,n)=\lim_{n\rightarrow\infty}\mathbf{P}(n)=2(\vartheta-\vartheta^{2}) $.
\end{itemize}
\end{teor}

\begin{proof}
\

\begin{itemize}
\item[i)] By Corollary \ref{C:dec}, we know that $\mathbf{P}(2)=1$ and that
$\mathbf{P}(n)$ is decreasing.

Now, if we recall Lemma \ref{L:in}, we can put
\[
P(r,n)=\frac{2r}{n}\left[  \left(  \frac{r}{n}-1\right)  +\sum_{i=r}%
^{n-1}\frac{1}{i}\right]  =\frac{2r}{n}\left(  \frac{r}{n}-1\right)
+\frac{2r}{n}\sum_{i=r}^{n-1}\frac{n}{i}\frac{1}{n},
\]
where the latter sum is a Riemann sum. Consequently,
\[
\lim_{n\rightarrow\infty}P(r,n)\simeq2x(x-1)+2x\int_{x}^{1}\frac{1}%
{t}dt=2x(x-1)-2x\log x=g(x)
\]
with $x=\displaystyle\lim_{n\rightarrow\infty}r/n$.

Finally, since $\mathbf{P}(n)$ is decreasing in $n$ and $2(\vartheta
-\vartheta^{2})$ is the maximum value of $g(x)$ due to Lemma \ref{L:max} it
follows that $\mathbf{P}(n)>2(\vartheta-\vartheta^{2})$ for every integer
$n>1$, as claimed.

\item[ii)] By Lemma \ref{L:max}, $g(x)$ reaches its maximum value at
$\vartheta$. Since $P(\cdot,n)$ reaches its maximum at $\mathfrak{m}(n)$ and
$x=\displaystyle \lim_{n\to\infty} r/n$, it follows immediately that
\[
\lim_{n\rightarrow\infty}\frac{\mathfrak{m}(n)}{n}=\vartheta.
\]

\item[iii)] From the previous work and Lemma \ref{L:max}, it is clear that
\[
\lim_{n\rightarrow\infty}P(\lfloor\vartheta n\rfloor,n)=\lim_{n\rightarrow
\infty}\mathbf{P}(n)=g(\vartheta)=2(\vartheta-\vartheta^{2})
\]
because $\displaystyle\lim_{n\to\infty} \lfloor\vartheta n\rfloor/n=\vartheta$.
\end{itemize}
\end{proof}

From the above theorem, it can be deduced that, for any value of $n$, the
optimal strategy achieves success with a probability greater than
$2(\vartheta-\vartheta^{2})=0.3238\dots$. Moreover, this is the approximate
value of the probability of success for large values of $n$, just by following
the simple strategy of rejecting the nearest integer to $-\frac{n}{2}%
W(-\frac{2}{e^{2}})=n\cdot0.2031\dots$. Thus, $[\vartheta\cdot n]$
constitutes a practical estimation of the optimal value of initial rejections,
just like $[n/e]$ does in the classical secretary problem. This is, in fact, true:
$[n\vartheta]$ is within 1 of the correct answer for all $n$, although many
errors ($20\%$) are made with this estimation:
\[
3, 8, 13, 18, 23, 32, 37, 42, 47, 52, 57, 62, 67, 72, 77, 82, 96, 101, 106,
\ 111, 116, 121,\dots
\]
However, it is also true that the error is negligible if compared with the
probability of the optimal strategy for large values of $n$.

Let us now look at a result that provides a better estimate
for $\mathfrak{m}(n)$ than the one above. However, let us first consider the
following auxiliary result where $\psi$ stands for the digamma function.

\begin{lem}
\label{lem:f} Let us consider the function $f(r,n):=-\frac{2r}{n}+\frac
{2r^{2}}{n^{2}}+\frac{2r}{n}\psi(n)-\frac{2r\log(r)}{n}+\frac{1}{n}$ and let
$\alpha(n)$ be the value for which the function $f(\cdot,n)$ reaches it
maximum in $[1,n]$. Then, $\alpha(n)\approx\mathfrak{m}(n)$.
\end{lem}
\begin{proof}
If we recall the definition of the digamma function $\psi$, we can write
\[
P(r,n)=-\frac{2r}{n}+\frac{2r^{2}}{n^{2}}+\frac{2r}{n}\psi(n)-\frac{2r}{n}%
\psi(r).
\]
Now, we know that for any integer $r$
\[
\psi(r)=\log r-\frac{1}{2r}+\epsilon(r),
\]
with $\epsilon(r)=\mathcal{O}(1/2r)$ if $r\rightarrow\infty$. Thus,
\[
P(r,n)=-\frac{2r}{n}+\frac{2r^{2}}{n^{2}}+\frac{2r}{n}\psi(n)-\frac{2r\log
(r)}{n}+\frac{1}{n}-\frac{2r\epsilon(r)}{n}=f(r,n)-\frac{2r\epsilon(r)}{n}.
\]

On the other hand, since $r\geq1$, it follows that there exists $k$ such that
$\left\vert 2r\epsilon(r)\right\vert \leq k$ for every $r$ and hence:
\[
f(r,n)-\frac{k}{n}\leq P(r,n)\leq f(r,n)+\frac{k}{n}%
\]

Obviously both functions $g_{1}(r,n):=f(r,n)-\frac{k}{n}$ and $g_{2}
(r,n):=f(r,n)+\frac{k}{n}$ reach their maximum at $\alpha(n)$. Now, let
$\beta_{1}(n)$ and $\beta_{2}(n)$ points such that $g_{2}(\beta_{1}%
(n),n)=g_{2}(\beta_{2}(n),n)=g_{1}(\alpha(n),n)$ and $\alpha(n)\in\left[
\beta_{1}(n),\beta_{2}(n)\right]  $. The inequality above implies that
\[
\beta_{1}(n)\leq\mathfrak{m}(n)\leq\beta_{2}(n)
\]

Finally, for every $r$ we have that $g_{2}(r,n)-g_{1}(r,n)=2k/n$ so it follows
that $\left\vert \beta_{2}(n)-\beta_{1}(n))\right\vert \underset
{n\rightarrow\infty}{\longrightarrow}0$ and, consequently also
\[
\left\vert \alpha(n)-\mathfrak{m}(n)\right\vert \underset{n\rightarrow\infty
}{\longrightarrow} 0.
\]
\end{proof}

\begin{teor}
With all the previous notation, the following hold:
\begin{itemize}
\item[i)] $\displaystyle\mathfrak{m}(n)\approx-\frac{n}{2} W\left(
-\frac{2e^{-2+\psi(n)}}{n}\right)  $, where $\psi$ denotes the digamma function.
\item[ii)] $\displaystyle \mathfrak{m}(n)\approx n \vartheta+\frac
{1}{4-2e^{2-2\vartheta}}$.
\end{itemize}
\end{teor}
\begin{proof}
\begin{itemize}
\item[i)] Since $f(\cdot,n)$ (see Lemma \ref{lem:f}) reaches its maximum at
$\alpha(n)$, it follows that
\[
0=\frac{\partial f}{\partial r}(\alpha(n),n)=-\frac{4}{n}+\frac{4\alpha
(n)}{n^{2}}+\frac{2}{n}\psi(n)-\frac{2\log(\alpha(n))}{n}.
\]

From this, and taking into account the definition of Lambert $W$ function it
follows that
\[
\alpha(n)=-\frac{n}{2}W(-\frac{2e^{-2+\psi(n)}}{n})
\]
so it is enough to apply Lemma \ref{lem:f}.

\item[ii)] Using i), we have that
\[
\lim_{n\rightarrow\infty}\left(  \mathfrak{m}(n)-n\vartheta\right)
=\lim_{n\rightarrow\infty}\left(  -\frac{n}{2}W(-\frac{2e^{-2+\psi(n)}}%
{n})-n\vartheta\right)  =\frac{1}{4-2e^{2-2\vartheta}}.
\]
Thus,
\[
\mathfrak{m}(n)\approx n\vartheta+\frac{1}{4-2e^{2-2\vartheta}}%
\]
as claimed.
\end{itemize}
\end{proof}

Using the formula $\left[n\vartheta+\frac{1}{4-2e^{2-2\vartheta}}\right]$
as far as our computational capacity let us, we find that the optimal cutoff value produces errors only for
very few values of $n$. Specifically, only four are known: 2, 3, 23 and 2971.
We now show another asymptotic result for the solution $\mathfrak{m}(n)\approx
h(n)$ that improves the result of the previous theorem in the sense that there
is no known value of $n>4$ for which $[h(n)]$ returns a wrong result.

\begin{prop}%
\[
\mathfrak{m}(n)\approx h(n):=\frac{1}{2} + \frac{3\,W(\frac{-2}{e^{2}})}{4} +
\frac{W(\frac{-2}{e^{2}})}{4\,\left(  1 + W(\frac{-2}{e^{2}}) \right)  } -
\frac{1}{2\,W(\frac{1}{e^{\frac{3}{2\,n}}\,n\,W(\frac{-2}{e^{2}})})}.
\]
\end{prop}
\begin{proof}
It suffices to verify that $h(n)$ has the following straight line as an
asymptote:
\[
n \vartheta+\frac{1}{4-2e^{2-2\vartheta}}.
\]

\end{proof}

Although no exceptions to the equality $\mathfrak{m}(n)=[h(n)]$ have been
found for $n>4$, there is no absolute guarantee that this occurs \textit{ad
infinitum}. This is what happens with the following asymptotic formula for the
optimal cutoff value in the classical secretary problem
\[
\left\lceil -\frac{1}{2W(-\frac{e^{-(1+1/(2n)}}{2n})}\right\rceil
\]
which coincides with the solution for all $n>3$ with no known exceptions.

\begin{rem}
The constant $\vartheta=-\frac{1}{2}W(-\frac{2}{e^{2}})=0.20318786\dots$
(A106533 in OEIS) appears in \cite{FER2} in the context of the secretary
problem considering a payoff of $(n-k+1)/n$. In fact, it appears erroneously
approximated as 0.20388, but is actually the constant which appears, as in
this study, as the solution of the equation $-\log(x)-2+2x=0$. Furthermore, as
a noteworthy curiosity, it should be pointed out that this constant has
appeared in a completely different context from the one addressed here as the
solution to the above equation (the Daley-Kendall model) and is known as the
\emph{rumour's constant} \cite{RUMOR,ru}.
\end{rem}

\subsection{What is the value of the information regarding the number of
objects?}

When the number of available objects is unknown, but it is known that it is a
uniform random variable $U[0,N]$, we have established that the asymptotic value
of the probability of success is $2(\vartheta-\vartheta^{2})=0.3238\dots$. It
is natural to ask how valuable exact knowledge of the number of objects that
we shall be able to observe may be. That is, if we win 1 Euro for making a
successful choice, how much might we pay to know the value of the random
variable $U[0,N]$ that represents the number of available objects.

If, prior to the inspections, we are given the information regarding the
number of objects, $n\in[1,N]$, in each case we will adopt the optimal
strategy set out in the first section of this paper that gives us a
probability of success
\[
\mathfrak{p}_{n}:=
\begin{cases}
\frac{n}{2(n-1)}, & \text{if $n>2$ is even};\\
\frac{n+1}{2n}, & \text{if $n>2$ is odd}\\
1 & \text{if $n\in\{1,2\}$ }%
\end{cases}
\]
and the overall probability of success will be $\mathfrak{P} _{N}:=N^{-1}%
\sum_{n=1}^{N}\mathfrak{p}_{n}$. Thus, it is straightforward to see that the
probability of success with a number of objects resulting from a uniform
random variable $[1,N]$, whose value is revealed before the inspections start,
when $N$ tends to infinity, is
\[
\lim_{N\rightarrow\infty}\mathfrak{P}_{N}=\frac{1}{2}%
\]
Therefore, the value of the information is $\mathfrak{P}_{N}-\mathbf{P}(N)$,
which, when $N$ tends to infinity, is
\[
\lim_{N\rightarrow\infty}(\mathfrak{P}_{N}-\mathbf{P}(N))=\frac{1}%
{2}-2(\vartheta-\vartheta^{2})=0.1761\dots
\]

\section{Unknown number of candidates with a $\lambda-$Poisson distribution}

For the classical secretary problem, Presman and Sonin, \cite{sonin} showed
that, if the number of candidates is Poisson distributed with parameter
$\lambda$, then the optimal stopping limit relation is $r^{\ast}%
(\lambda)/\lambda\rightarrow e^{-1}$ and that this is also the asymptotic
value of the probability of success. In the variant we have introduced in this
paper, an analogous conclusion is reached by replacing $1/e$ by $1/2$, as seen
in Section 2.

If the number of objects in the best or worst variant is Poisson
distributed with parameter $\lambda$, the probability of success after
rejecting the first $r$ candidates will be
\[
P(0,\lambda):=\frac{\lambda}{e^{\lambda}} + \frac{\lambda^{2}}{2\,e^{\lambda}%
}+\sum_{i=3}^{\infty}\frac{2}{i}\frac{e^{-\lambda}\lambda^{i}}{i!},
\]
\[
P(1 ,\lambda):= \frac{\lambda^{2}}{2\,e^{\lambda}}+\sum_{i=3}^{\infty}\frac{2
}{i}\frac{e^{-\lambda}\lambda^{i}}{i!},
\]
\[
P(r,\lambda):=\sum_{i=r+1}^{\infty}\frac{2(i-r)r}{i(i-1)}\frac{e^{-\lambda
}\lambda^{i}}{i!}
\]

Following the same notation as in Section 3, we denote by $\mathfrak{m}%
(\lambda)$ the value for which $P(\cdot,\lambda)$ reaches its maximum value
and by $\mathbf{P}(\lambda)$, the maximum value it reaches; i.e.
$\mathbf{P}(\lambda):=P(\mathfrak{m}(\lambda),\lambda)$. Unlike what occurred
in the previous case, the probability of success in the game, $\mathbf{P}%
(\lambda)$, is not monotone with respect to $\lambda$. Next, we show the graph
of $\mathbf{P}(\lambda)$ (Figure 1), which is seen to comprise concave arcs in
each interval $[\lambda_{i},\lambda_{i+1}]$, where $\lambda_{0}=0$,
$\lambda_{1}$ is such that $P(0,\lambda_{1})=P(2,\lambda_{1})$ and
$\lambda_{i}$ for $i>1$ is such that $P(i,\lambda_{i})=P(i+1,\lambda_{i})$.%

\begin{figure}
{\includegraphics[
height=2.7354in,
width=4.3673in
]{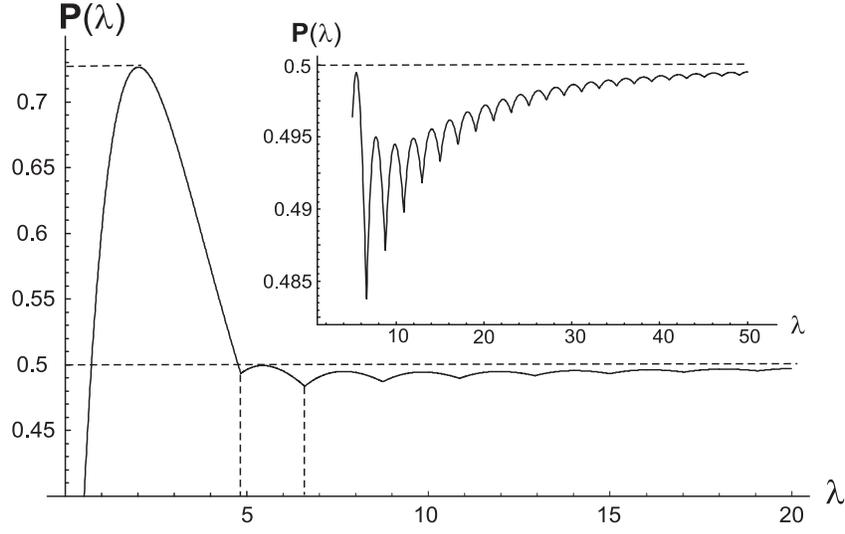}}\\
\caption{The probability $P(\lambda)$.}
\end{figure}

The maximum probability of success is achieved in $\lambda=\widetilde{\lambda}
:=2.01771\dots$, the solution to the equation
\[
0=-2 + 2\,e^{x} - x + 2\, \gamma\,x + x^{2} + 2\,x\,\Gamma(-x) +
2\,x\,\log(-x),
\]
where $\mathfrak{m}( \widetilde{\lambda} ) =0$ and $\mathbf{P}(\widetilde
{\lambda} )= 0.726470\dots$

We will proof that, as Figure 1 suggests, $\mathbf{P}(\lambda)$ tends to $1/2$
when $\lambda$ grows to infinity and that this propability can be
asymptotically reached using $\frac{\lambda}{2}$ as cutoff value. But first we
need some technical results.

\begin{lem}
\label{LEM:f} Let us define the function
\[
f(r,\lambda):=\sum_{i=2}^{r}\frac{2(i-r)r}{i(i-1)}\frac{e^{-\lambda}%
\lambda^{i}}{i!}.
\]
Then, $\lim_{\lambda\rightarrow\infty}f(\frac{\lambda}{2},\lambda)=0$.
\end{lem}

\begin{proof}
First of all, note that
\[
\left\vert \sum_{i=2}^{\lambda/2}\frac{2(i-\lambda/2)\lambda/2}{i(i-1)}%
\frac{e^{-\lambda}\lambda^{i}}{i!}\right\vert <\lambda^{2}e^{-\lambda}%
\sum_{i=2}^{\lambda/2}\frac{\lambda^{i}}{i!} <\frac{1}{\lambda}
\]
so, if we define
\[
a_{n}:=2^{2}n^{2}e^{-2n}\sum_{i=2}^{n}\frac{2^{i}n^{i}}{i!},
\]
we just need to prove that $\lim_{n\rightarrow\infty}a_{n}=0$.

Now,
\[
a_{n+1}=2^{2}(n+1)^{2}e^{-2(n+1)}\sum_{i=2}^{n+1}\frac{2^{i}(n+1)^{i}}{i!}
\]
and since
\begin{align*}
\sum_{i=2}^{n+1}\frac{2^{i}(n+1)^{i}}{i!}  &  =\sum_{i=2}^{n}\frac{2^{i}n^{i}%
}{i!}+\frac{2^{2}}{2!}(2n+1)+\frac{2^{3}}{3!}(3n^{2}+3n+1)+\ldots+\\
&  +\frac{2^{n}}{n!}\left(  n^{n}+\frac{n}{1!}n^{n-1}+\frac{n(n-1)}{2!}%
n^{n-2}+\ldots+\frac{n(n-1)\cdots2.1}{n!}\right)  +\frac{2^{n+1}}%
{(n+1)!}(n+1)^{n+1}%
\end{align*}
we have that
\begin{align*}
a_{n+1}  &  =2^{2}(n+1)^{2}e^{-2(n+1)}\sum_{i=2}^{n}\frac{2^{i}n^{i}}%
{i!}+2^{2}(n+1)^{2}e^{-2(n+1)}\frac{2^{2}}{2!}(2n+1)+\ldots+\\
&  +2^{2}(n+1)^{2}e^{-2(n+1)}\frac{2^{n}}{n!}\left(  n^{n}+\frac{n}{1!}%
n^{n-1}+\frac{n(n-1)}{2!}n^{n-2}+\ldots+\frac{n(n-1)\cdots2.1}{n!}\right) \\
&  +2^{2}(n+1)^{2}e^{-2(n+1)}\frac{2^{n+1}}{(n+1)!}(n+1)^{n+1}%
\end{align*}

On the other hand,
\[
\sum_{i=2}^{n}\frac{2^{i}n^{i}}{i!}=\frac{a_{n}}{2^{2}n^{2}e^{-2n}}.
\]
Hence
\begin{align*}
a_{n+1}  &  =\left(  \frac{n+1}{n}\right)  ^{2}\frac{a_{n}}{e^{2}}%
+2^{2}(n+1)^{2}e^{-2(n+1)}\frac{2^{2}}{2!}(2n+1)+\ldots+\\
&  +2^{2}(n+1)^{2}e^{-2(n+1)}\frac{2^{n}}{n!}\left(  n^{n}+\frac{n}{1!}%
n^{n-1}+\frac{n(n-1)}{2!}n^{n-2}+\ldots+\frac{n(n-1)\cdots2.1}{n!}\right) \\
&  +2^{2}(n+1)^{2}e^{-2(n+1)}\frac{2^{n+1}}{(n+1)!}(n+1)^{n+1},
\end{align*}
which clearly implies that
\[
\lim_{n\rightarrow\infty} \left[  a_{n+1}-\left(  \frac{n+1}{n}\right)
^{2}\frac{a_{n}}{e^{2}}\right]  =0.
\]

Finally, $a_{n}$ being decreasing and non-negative, we have that
$\lim_{n\rightarrow\infty}a_{n+1}=\lim_{n\rightarrow\infty}a_{n}=l$. Thus,
\[
l=\frac{l}{e^{2}}\Rightarrow l\left(  1-\frac{1}{e^{2}}\right)  =0\Rightarrow
l=0
\]
and
\[
l= \lim_{n\rightarrow\infty}a_{n}=0
\]
as claimed.
\end{proof}

\begin{teor}
\label{TEOR:PE} Consider the function $P^{\star}$ defined over $\mathbb{N}%
\times\mathbb{R}^{+}$ and given by
\[
P^{\star}(r,\lambda):=\sum_{i=2}^{\infty}\frac{2(i-r)r}{i(i-1)}\frac
{e^{-\lambda}\lambda^{i}}{i!}.
\]
For a fixed $\lambda$, let us denote by $\mathcal{M}^{\star}(\lambda)$ the
value for which $P^{\star} (\cdot,\lambda)$ reaches its maximum and put
$\mathcal{P}^{\star}(\lambda):=P^{\star}(\mathcal{M}^{\star}(\lambda
),\lambda)$. Then, the following hold:

\begin{itemize}
\item[i)] $\displaystyle\lim_{\lambda\rightarrow\infty}\frac{\mathcal{M}%
^{\star}(\lambda)}{\lambda}=\frac{1}{2}$.

\item[ii)] $\displaystyle \lim_{\lambda\rightarrow\infty}P^{\star}%
(\frac{\lambda}{2},\lambda)=\frac{1}{2}$.

\item[iii)] $\mathcal{M}^{\star}(\lambda)\approx\frac{\lambda}{2}-1$.
\end{itemize}
\end{teor}

\begin{proof}
First of all, we are going to compute the value of $\mathcal{M}^{\star
}(\lambda)$. To do so, let us consider functions
\begin{align*}
S_{1}(x)  &  =\sum_{i=2}^{\infty}\frac{x^{i}}{(i-1)i!},\\
S_{2}(x)  &  =\sum_{i=2}^{\infty}\frac{x^{i}}{i(i-1)i!}=-\sum_{i=2}^{\infty
}\frac{x^{i}}{i\text{ }i!}+S_{1}(x),\\
S_{3}(x)  &  =\sum_{i=2}^{\infty}\frac{x^{i}}{i\text{ }i!}%
\end{align*}
Then, we have that
\[
S_{1}^{\prime\prime}(x)=\frac{e^{x}-1}{x},\quad S_{3}^{\prime}(x)=\frac
{e^{x}-1}{x}-1.
\]
If we integrate these expressions we obtain that
\begin{align*}
S_{1}(x)  &  =1-e^{x}+x-\gamma x+x\operatorname{E}(x)-x\log x\\
S_{3}(x)  &  =-\gamma+\operatorname{E}(x)-\log x-x
\end{align*}
where
\[
\operatorname{E}(\lambda)=\gamma+\log\lambda+\int_{0}^{\lambda} \frac{e^{x}%
-1}{x}dx.
\]
Consequently,
\[
S_{2}(x)=1-e^{x}+2x+\left(  \gamma-\operatorname{E}(x)+\log x\right)  \left(
1-x\right)  .
\]
Finally, note that
\[
P^{\star}(r,\lambda)=2re^{-\lambda}\left[  \sum_{i=2}^{\infty}\frac
{\lambda^{i}}{(i-1)i!}-r\sum_{i=2}^{\infty}\frac{\lambda^{i}}{i(i-1)i!}%
\right]  =2re^{-\lambda}S_{1}(\lambda)-2r^{2}e^{-\lambda}S_{2}(\lambda)
\]
Hence,
\[
\frac{\partial P^{\star}(r,\lambda)}{\partial r}=2e^{-\lambda}S_{1}%
(\lambda)-4re^{-\lambda}S_{2}(\lambda)
\]
and, consequently, if we define
\[
r(\lambda):=\frac{1-e^{\lambda}+\lambda-\gamma\lambda+\lambda\operatorname{E}%
(\lambda)-\lambda\log\lambda}{2\,\left[  1-e^{\lambda}+2\lambda+\left(
\gamma-\operatorname{E}(\lambda)+\log\lambda\right)  (1-\lambda)\right]  }%
\]
it is clear that $0=\frac{\partial P^{\star}(r,\lambda)}{\partial
r}|_{r=r(\lambda)}$. Thus, we have just obtained that $\mathcal{M}^{\star
}(\lambda)=r(\lambda)$.

Once we have computed the value of $\mathcal{M}^{\star}(\lambda)$, we are in
the condition to prove the three statements of the theorem

\begin{itemize}
\item[i)] We have
\[
\lim_{\lambda\rightarrow\infty}\frac{\mathcal{M}^{\star}(\lambda)}{\lambda
}=\frac{1}{2}\lim_{\lambda\rightarrow\infty}\frac{1-e^{\lambda}+\lambda
-\gamma\lambda+\lambda\operatorname{E}(\lambda)-\lambda\log\lambda}%
{\lambda\,\left[  1-e^{\lambda}+2\lambda+\left(  \gamma-\operatorname{E}%
(\lambda)+\log\lambda\right)  (1-\lambda)\right]  }
\]
so, if we recall the definition of $\operatorname{E}(\lambda)$ we have that
\[
\lim_{\lambda\rightarrow\infty}\frac{\mathcal{M}^{\star}(\lambda)}{\lambda
}=\frac{1}{2}\lim_{\lambda\rightarrow\infty}\frac{1-e^{\lambda}+\lambda
+\lambda\int_{0}^{\lambda}\frac{e^{x}-1}{x}dx}{\lambda\,\left[  1-e^{\lambda
}+2\lambda+(\lambda-1)\int_{0}^{\lambda}\frac{e^{x}-1}{x}dx\right]  }%
\]
and, applying L'Hôpital's rule repeatedly we obtain:
\[
\lim_{\lambda\rightarrow\infty}\frac{\mathcal{M}^{\star}(\lambda)}{\lambda
}=\frac{1}{2}\lim_{\lambda\rightarrow\infty}\frac{1+\lambda}{\lambda
\,+3}=\frac{1}{2},
\]
as claimed.

\item[ii)] Now, we have that
\[
\lim_{\lambda\rightarrow\infty}P^{\star}(\frac{\lambda}{2},\lambda)
=\lim_{\lambda\rightarrow\infty}\frac{\lambda}{2e^{\lambda}}\left[
(-\lambda^{2}+3\lambda)\int_{0}^{\lambda}\frac{e^{x}-1}{x}dx+(\lambda
-2)e^{\lambda}-2\lambda^{2}+\lambda+1)\right]
\]
so, again by L'Hôpital's rule, we obtain:
\[
\lim_{\lambda\rightarrow\infty}P^{\star}(\frac{\lambda}{2},\lambda)=\frac
{1}{2}\lim_{\lambda\rightarrow\infty}\frac{6+2\lambda}{2\lambda\,+5}=\frac
{1}{2}.
\]

\item[iii)] Since
\[
\lim_{\lambda\rightarrow\infty}(\mathcal{M}^{\star}(\lambda)-\lambda
/2)=\frac{1}{2}\lim_{\lambda\rightarrow\infty}\frac{1-e^{\lambda}+\lambda
e^{\lambda}-2\lambda^{2}+(2\lambda-\lambda^{2})\int_{0}^{\lambda}\frac
{e^{x}-1}{x}dx}{1-e^{\lambda}+2\lambda+(\lambda-1)\int_{0}^{\lambda}%
\frac{e^{x}-1}{x}dx},
\]
L'Hôpital's rule leads to
\[
\lim_{\lambda\rightarrow\infty}(\mathcal{M}^{\star}(\lambda)-\lambda
/2)=\frac{1}{2}\lim_{\lambda\rightarrow\infty}\frac{-2\lambda-4}{\lambda}=-1.
\]
Thus,
\[
\mathcal{M}^{\star}(\lambda)\approx\frac{\lambda}{2}-1,
\]
as claimed.
\end{itemize}
\end{proof}

\begin{lem}
\label{lem:p} Assume that the number of objects $n$ is Poisson distributed with parameter $\lambda$. If the (random) value
of $n$ is revealed before the inspections start, then the probability of
success tends to $1/2$ when $\lambda$ tends to infinity. In other words:
\[
\lim_{\lambda\rightarrow\infty}\sum_{n=1}^{\infty}\mathfrak{p}_{n}%
\frac{e^{-\lambda}\lambda^{n}}{n!}=\frac{1}{2}.
\]

\end{lem}

\begin{proof}
We reason just like in Subsection 3.1. If, prior to the inspections, we are
given the information regarding the number of objects, $n\in\lbrack0,\infty)$,
in each case we will adopt the optimal strategy set out in the first section
of this paper that gives us a probability of success $\mathfrak{p}_{n}$ and
the overall probability of success will be
\[
\mathfrak{P}_{\lambda}:=\sum_{n=1}^{\infty}\mathfrak{p}_{n}\frac{e^{-\lambda
}\lambda^{n}}{n!}=2\sum_{n=1}^{\infty}\frac{n}{-1+2\,n}\frac{e^{-\lambda
}\lambda^{n}}{n!}=\frac{\lambda\,\mathrm{S}(\lambda)}{e^{\lambda}},
\]
where
\[
\mathrm{S}(\lambda)=\int_{0}^{\lambda}\frac{\sinh(x)}{x}dx.
\]
so we have that
\[
\lim_{\lambda\rightarrow\infty}\mathfrak{P}_{\lambda}=\frac{1}{2}.
\]

\end{proof}

\begin{teor}
In the previous setting we have that
\[
\lim_{\lambda\rightarrow\infty}P(\frac{\lambda}{2},\lambda)=\lim
_{\lambda\rightarrow\infty}P(\mathfrak{m}(\lambda),\lambda)=\lim
_{\lambda\rightarrow\infty}\mathbf{P}(\lambda)=\frac{1}{2}.
\]

\end{teor}

\begin{proof}
Since $\mathfrak{p}_{n}$ (see Lemma \ref{lem:p}) is the maximum value of the
probability of suceeding in the problem with complete information with $n$
objects, it follows for all $n>1$ that
\[
\mathfrak{p}_{n} \geq P_{n}(\mathfrak{m}(\lambda))=\frac{2(n-\mathfrak{m}%
(\lambda))\mathfrak{m}(\lambda)}{n(n-1)}%
\]
and hence,
\[
\sum_{n=1}^{\infty}\mathfrak{p}_{n}\frac{e^{-\lambda}\lambda^{n}}{n!} \geq
\sum_{n=\mathfrak{m(\lambda)}+1}^{\infty}\frac{2(n-\mathfrak{m}(\lambda
))\mathfrak{m}(\lambda)}{n(n-1)}\frac{e^{-\lambda}\lambda^{n}}{n!}=
\mathbf{P}(\lambda)
\]
So, if we take upper limits it is clear that $\overline{\lim}_{\lambda
\rightarrow\infty}P(\mathfrak{m}(\lambda),\lambda) \leq\frac{1}{2}$.

Now, recalling the function $f$ from Lemma \ref{LEM:f} we have that
\[
\mathbf{P}(\lambda)=P(\mathfrak{m}(\lambda),\lambda)\geq P(\lambda
/2,\lambda)=P^{\star}(\lambda/2,\lambda)-f(\lambda/2,\lambda).
\]

But, due to Lemma \ref{LEM:f} and Theorem \ref{TEOR:PE} we obtain that
\[
\lim_{\lambda\rightarrow\infty}P^{*}(\lambda/2,\lambda)-f(\lambda
/2,\lambda)=\frac{1}{2}%
\]
and taking lower limits:
\[
\underline{ \lim}_{\lambda\rightarrow\infty}\mathbf{P}(\lambda)\geq\frac{1}%
{2}.
\]

In conclusion,
\[
\frac{1}{2}\leq\underline{ \lim}_{\lambda\rightarrow\infty}\mathbf{P}%
(\lambda)\leq\overline{\lim}_{\lambda\rightarrow\infty}\mathbf{P}(\lambda)
\leq\frac{1}{2}%
\]
and the proof in complete.
\end{proof}

This theorem satisfactorily solves the problem, since it implies that
$\lambda/2$ is the best possible estimation for the cutoff value because it
guaratees an asymptotical succeding probability of $\frac{1}{2}$ (exactly the
same that the exact value of $\mathfrak{m}(\lambda)$ would provide). Thus,
this result proves that if the number of objects is unknown with a $\lambda
-$Poisson distribution, the strategy that we must follow in order to maximize
the succeding probability is to select the first object which is better or
worst than the previous ones just after rejecting the first $\lambda/2$
objects. With this strategy we will succeed approximately one half of the times.

Unlike what happened in the previous section (see 3.1.), when facing with an
unknown number of objects resulting from a Poisson distribution with parameter
$\lambda$, the value of the information of the number of objects is
asymptotically null. For large values of $\lambda$, the difference between the
probability of success with and without this information is negligible.

In addition, as far as our computing capabilities let us check, $\lfloor
\lambda/2-1 \rfloor$ coincides with the exact value $\mathfrak{m}(\lambda)$
for every integral value of $\lambda$ greater than 1. Hence, the following conjecture.

\begin{con}
$\mathfrak{m}(\lambda) \approx\mathcal{M}^{\star}(\lambda)\approx\frac
{\lambda}{2}-1$.
\end{con}

\section{Payoff proportional to the number of performed or non-performed
interviews}

We now consider the same underlying game as in the previous sections, but
introducing different payoffs based on the number of observed objects. We
consider that, in the case of success (choosing the best or the worst object),
we shall receive a payoff based on the number of interviews carried out. The
aim is therefore to maximize the payoff. Although the goal is still to choose
the best or worst object, there are incentives (or penalties) for carrying out
more (or fewer) inspections of the objects than those that maximize the
probability of success found in Section 2, as higher payoffs will be obtained.
We first consider the case in which the payoff received is proportional to the
number of observed objects and then the case in which the payoff received is
proportional to the number of non-observed objects.

If we consider that the payoff for success in the game with $n$ available
objects after having carried out $k$ inspections is $g(k,n)$, the expected
payoff following the strategy of rejecting $r$ objects and choosing the first
observed object that is better or worse than those seen so far will thus be:%

\[
\mathbf{G}(r,n):=\frac{2\binom{r}{2}}{n} \sum_{k = r + 1}^{n} \frac
{g(k,n)}{\binom{k-1}{2}}%
\]
We consider two cases: the first, where $g(k,n)=\frac{k}{n}$, and the second,
where $g(k,n)=\frac{n-k}{n}$. In both cases, we obtain asymptotic formulas for
the optimal cutoff value and the expected payoff following the optimal strategy.

\subsection{Payoff proportional to the number of performed interviews}

\begin{teor}
Consider the function
\[
\mathbf{G}(r,n):=\frac{2\binom{r}{2}}{n^{2}} \sum_{k = r + 1}^{n} \frac
{k}{\binom{k-1}{2}}%
\]
defined for every pair of integers $(r,n)$ such that $1<r<n$. Let us denote by
$\mathcal{M}(n)$ the value for which the function $\mathbf{G}(\cdot,n)$
reaches its maximun. Then,

\begin{itemize}
\item[i)] $\displaystyle \lim_{n\rightarrow\infty}\frac{\mathcal{M}(n)}
{n}=\frac{1}{\sqrt{e}}$; i.e, $\mathcal{M}(n)\sim\frac{n}{\sqrt{e}}$.

\item[ii)] $\displaystyle \lim_{n\rightarrow\infty}\mathbf{G}\left(  \left[
\frac{n}{\sqrt{e}}\right]  ,n\right)  =\frac{1}{e}$.
\end{itemize}
\end{teor}

\begin{proof}
First of all, observe that
\[
\mathbf{G}(r,n)=\frac{2\binom{r}{2}}{n^{2}}\sum_{k=r+1}^{n}\frac{k}%
{\binom{k-1}{2}}=\frac{r(r-1)}{n^{2}}\sum_{k=r+1}^{n}\frac{2k}{(k-1)(k-2)}
\]
Now, using telescopic sums, we have that
\[
\sum_{k=r+1}^{n}\frac{2k}{(k-1)(k-2)}=2\sum_{k=r+1}^{n}\left(  \frac{2}%
{k-2}-\frac{1}{k-1}\right)  =2\left(  \frac{2}{r-1}-\frac{1}{n-1}\right)
+2\sum_{k=r+1}^{n-1}\frac{1}{k-1}%
\]
Thus,
\begin{align*}
\mathbf{G}(r,n)  &  =\frac{2r(r-1)}{n^{2}}\left[  \frac{2}{r-1}-\frac{2}%
{n-1}+\sum_{i=r}^{n-1}\frac{1}{i}\right]  =\frac{4r(n-r)}{n^{2}(n-1)}%
+\frac{2r(r-1)}{n^{2}}\sum_{i=r}^{n-1}\frac{1}{i}=\\
&  =\frac{4}{n-1}\frac{r}{n}\left(  1-\frac{r}{n}\right)  +2\frac{r}{n}\left(
\frac{r}{n}-\frac{1}{n}\right)  \sum_{i=r}^{n-1}\frac{n}{i}\frac{1}{n}%
\end{align*}
where the latter sum is a Riemann sum. Consequently,
\[
\lim_{n\rightarrow\infty}\mathbf{G}(r,n)\simeq2x^{2}\int_{x}^{1}\frac{1}%
{t}dt=-2x^{2}\log x=g(x)
\]
with $x=\displaystyle\lim_{n\rightarrow\infty}r/n$.

Now, it is an elementary Calculus exercise to see that the function
$g(x)=-2x^{2}\log x$ reaches it absolute maximum at the point
\[
\frac{1}{\sqrt{e}}=0.606531\dots
\]
Moreover, the value of this maximum is:
\[
g\left(  \frac{1}{\sqrt{e}}\right)  =\frac{1}{e}=0.367879\dots
\]

After this previous work we can proof the statements of the theorem.

\begin{itemize}
\item[i)] Since $\mathbf{G}(\cdot,n)$ reaches its maximum at $\mathcal{M}(n)$
and $x=\displaystyle\lim_{n\rightarrow\infty}r/n$, it follows immediately
that
\[
\lim_{n\rightarrow\infty}\frac{\mathcal{M}(n)}{n}=\frac{1}{\sqrt{e}}%
\]

\item[ii)] It is clear that
\[
\lim_{n\rightarrow\infty}\mathbf{G}\left(  \left[  \frac{n}{\sqrt{e}}\right]
,n\right)  =g\left(  \frac{1}{\sqrt{e}}\right)  =\frac{1}{e}%
\]
because $\displaystyle\lim_{n\rightarrow\infty}\left[  \frac{n}{\sqrt{e}
}\right]  /n=\frac{1}{\sqrt{e}}$.
\end{itemize}
\end{proof}

Recall that in Theorem 2 iv), an asymptotic formula was obtained for the
solution $\mathfrak{m}(n)=n\vartheta+\mathcal{O}(1)$ taking a series
expansion. This was possible because equations of type $0=A+Bx+C\log(x)$
are solvable in terms of Lambert $W$ Function. In this case, the equation that
appears is of the type $A+Bx+C\log(x)+Dx\log(x)=0$, for which there is no
known solution that can be expressed symbolically. Nonetheless, the
simplification of this equation achieved by replacing $r/n$ by its asymptotic
value, $1/\sqrt{e}$, gives rise to a formula that improves the estimation
$[n/\sqrt{e}]$ and, in this case, for $\mathcal{M}(n)$, we conjecture

\begin{con}%
\[
\mathcal{M}(n)\approx\frac{n}{\sqrt{e}} + \mu,
\]
where $\mu:= 2 - \frac{5}{2\,\sqrt{e}} = 0.483673350\dots$
\end{con}

In fact, $[\frac{n}{\sqrt{e}}+\mu]$ is the value of $\mathcal{M}(n)$ for all
$n$ up to 10000, with the only exceptions of $n\in\{33,94\}$.

\subsection{Payoff proportional to the number of non-performed interviews}

\begin{teor}
Consider the function
\[
G(r,n):=\frac{2\binom{r}{2}}{n^{2}}\sum_{k=r+1}^{n}\frac{n-k}{\binom{k-1}{2}}%
\]
defined for every pair of integers $(r,n)$ such that $1<r<n$ and let
\[
\theta:=-\frac{1}{2W_{_{-1}}(-\frac{1}{2\sqrt{e}})}=e^{\frac{1}{2} +
W_{-1}(\frac{-1}{2\,\sqrt{e}})}%
\]
be the solution to the equation $2x\log(x)=x-1$. Denote by $\mathcal{M}(n)$
the value for which the function $G(\cdot,n)$ reaches its maximum. Then the
following hold:
\begin{itemize}
\item[i)]
$\displaystyle\lim_{n\rightarrow\infty}\frac{\mathcal{M}(n)}{n}=\theta=
0.284668\dots$; i.e, $\mathcal{M}( n)\sim n \cdot\theta$.

\item[ii)]
$\displaystyle \lim_{n\rightarrow\infty}G([ n \cdot\theta],n)=\theta(1-\theta)=0.2036321\dots$
\end{itemize}
\end{teor}

\begin{proof}
We reason like in the previous theorem. First, observe that
\begin{align*}
G(r,n)  &  =\frac{2r(r-1)}{n^{2}}\sum_{k=r+1}^{n}\frac{(n-k)}{(k-1)(k-2)}%
=\frac{2r(r-1)}{n^{2}}\left[  \frac{n-2}{r-1}-\frac{n}{n-1}+\sum_{i=r}%
^{n-1}\frac{1}{i}\right]  =\\
&  =2\frac{r}{n}\left(  1-\frac{2}{n}\right)  -2\frac{r}{n}\left(  \frac
{r}{n-1}-\frac{1}{n-1}\right)  -2\frac{r}{n}\left(  \frac{r}{n}-\frac{1}%
{n}\right)  \sum_{i=r}^{n-1}\frac{n}{i}\frac{1}{n},
\end{align*}
where the latter sum is a Riemann sum. Consequently,
\[
\lim_{n\rightarrow\infty}G(r,n)\simeq2x-2x^{2}-2x^{2}\int_{x}^{1}\frac{1}%
{t}dt=2x\left(  1-x+x\log x\right)  =g(x)
\]
with $x=\displaystyle\lim_{n\rightarrow\infty}r/n$. It is easily seen that
$g(x)=2x\left(  1-x+x\log x\right)  $ reaches it absolute maximum at the
point
\[
\theta=-\frac{1}{2W_{-1}(-\frac{1}{2\sqrt{e}})}=0.284668\dots
\]
Moreover, the value of this maximum is:
\[
g\left(  \theta\right)  =\theta- {\theta}^{2}= 0.203632\dots
\]

Now, we prove the statements of the theorem.

\begin{itemize}
\item[i)] Since $G(\cdot,n)$ reaches its maximum at $\mathcal{M}(n)$ and
$x=\displaystyle\lim_{n\rightarrow\infty}r/n$, it follows immediately that
\[
\lim_{n\rightarrow\infty}\frac{\mathcal{M}(n)}{n}=\theta
\]

\item[ii)] It is clear that
\[
\lim_{n\rightarrow\infty}G\left(  \left[  n\cdot\theta\right]  ,n\right)
=g\left(  \theta\right)  =\theta(1-\theta)
\]
because $\displaystyle\lim_{n\rightarrow\infty}\left[  n\cdot\theta\right]
/n=\theta$.
\end{itemize}
\end{proof}

Just like in the previous case, it has not been possible to determine the
asymptotic formula of the type $\mathcal{M}(n)\approx n\cdot\theta+\beta$.
Nevertheless, it is possible to conjecture that $\mathcal{M}(n)\approx
n\cdot\theta+\mathcal{O}(1)$. Moreover, we conjecture that $\mathcal{M}%
(n)\approx n\cdot\theta+\mu$, where $\mu\cong1.4034$ is an approximation that
has been experimentally found on the basis of the calculation of the solution
for large values of $n$. In this case, $[n\cdot\theta+\mu]$ is the exact value
of $\mathcal{M}(n)$ for all values of $n$ up to 10000, without any exception.

\begin{rem}
The constant $\theta=-\frac{1}{2W_{-1}(-\frac{1}{2\sqrt{e}})}=0.284668\dots$
also appears related to rumour theory \cite{RUMOR,ru} and to
Gabriel's Horn (see A101314 in OEIS).
\end{rem}

\section{Conclusions and future work}

In this paper, we have analyzed a very natural variant of the secretary
problem that has probably not received much attention because its asymptotic
analysis for the game with full information is very simple. However, its study
has been found to be of interest in the case of an unknown random number of objects
and with payoffs that are proportional to the number of inspected objects. We
show a comparative table of the asymptotic optimal stopping rule
$\mathfrak{m}(\cdot)$ and the $\mathbb{E}(\text{payoff})$ (the asymptotic mean
payoff) in the classical problem and in the best/worst variant studied in this paper.

All the constants that appear in the table can be expressed in terms of the
following three constants:
\[
e=2.71828\dots,\ \vartheta:=\frac{W(-2 e^{-2})}{2}=0.20318\dots,
\]
\[
\theta:=-\frac{1}{2W_{-1}(-\frac{1}{2\sqrt{e}})}=0.28466\dots.
\]
\[%
\begin{tabular}
[c]{|c|c|c|c|c|}\hline
& \multicolumn{2}{|c|}{Classic} & \multicolumn{2}{|c|}{Best or Worst}%
\\\cline{2-5}
& $\mathfrak{m}(\cdot)$ & $\mathbb{E}(\text{payoff})$ & $\mathfrak{m}(\cdot)$
& $\mathbb{E}(\text{payoff})$\\\hline
$n$ objects & $ne^{-1}$ & $e^{-1}$ & $n/2$ & 1/2\\\hline
$\lambda-$Poisson & $\lambda/e$ & $e^{-1}$ & $\lambda/2$ & 1/2\\\hline
Uniform $[1,n]$ & $ne^{-1/2}$ & $2e^{-2}$ & $n\vartheta$ & $2(\vartheta
-\vartheta^{2})$\\\hline
Payoff $=(n-k)/n$ & $n\vartheta$ & $\vartheta-\vartheta^{2}$ & $n\theta$ &
$\theta-{\theta}^{2}$\\\hline
Payoff $=k/n$ & $n/2$ & $1/4$ & $ne^{-1/2}$ & $e^{-1}$\\\hline
\end{tabular}
\]

Although we have not addressed its study, it is not unreasonable to conjecture
that the counterpart of the $1/e$-law of best choice \cite{law} in this game
is what we might call the $1/2$-law. Let us suppose that, independently of
each other, all applicants have the same arrival time density f on $[0,T]$ and
let F denote the corresponding arrival time distribution function, i.e.%

\[
F(t) = \int_{0}^{t} f(s)ds , \, 0 \le t \le T.
\]

The 1/2-law: Let $\tau$ be such that $F(\tau)=1/2$. Consider the strategy of
waiting and observing all applicants up to time $\tau$ and then choosing, if
possible, the first candidate after time $\tau$ who is better than all the
preceding ones.

Thus, this strategy, which we shall call the 1/2-strategy, has the following properties:

\begin{itemize}
\item[i)] It yields for all N a success probability of at least 1/2,
\item[ii)] It is the only strategy guaranteeing this lower success probability
bound, 1/2, and the bound is optimal,
\item[iii)] It chooses, if there is at least one applicant, none at all with a
probability of exactly 1/2.
\end{itemize}

Finally, we wish to draw attention a question that has not been solved in this
paper, namely that of how to establish the value of the constants that appear
in the asymptotic formula of the optimal cutoff of the problems posed in
Section 4. Although from a practical point of view, an approximation of these
constants with 4 digit accuracy seems to provide the correct solution (except,
perhaps, for some isolated values of $n$), it would be very interesting to
obtain them exactly as a solution of certain transcendental equations, as was
possible in Section 2 with the constant $\frac{1}{4-2e^{2-2\vartheta}}$, where
$\vartheta$ represents the rumour's constant.


\begin{thebibliography}{99}


\bibitem{KK} J.N. Bearden. \newblock A new secretary problem with rank-based selection and cardinal
payoffs. \newblock {\em Journal of Mathematical Psychology}, 50: 58-59. 2006.
                                                                                              %
\bibitem {law}F. Th. Bruss. \newblock A unified Approach to a Class of Best
Choice problems with an Unknown Number of Options.
\newblock {\em Annals of Probability}, 12(3): 882-891. 1984.

\bibitem {lam}R. M. Corless, G. H. Gonnet, D. E. G. Hare, D. J. Jeffrey, and
D. E. Knuth. \newblock On the Lambert W function.
\newblock {\em Adv. Comput. Math.}, 5: 329-359. 1996.

\bibitem {RUMOR}D.J. Daley and D.G. Kendall. \newblock Stochastic rumours.
\newblock {\em Journal of the Institute of Mathematics and Its Applications},
1:42-55. 1965.

\bibitem {48}E. B. Dynkin. \newblock The optimum choice of the instant for
stopping a markov process. \newblock {\em Soviet Mathematics - Doklady},
4:627-629, 1963.

\bibitem {FER}T.S. Ferguson. \newblock Who solved the secretary problem?
\newblock {\em Statistical Science}, 4(3): 282-296. 1989.

\bibitem {FER2}T. S. Ferguson, J. P. Hardwick and M. Tamaki.
\newblock Maximizing the duration of owning a relatively best object.
\newblock {\em Contemporary Mathematics:
Strategies for Sequential Search and Selection in Real Time, American
Mathematics Association  (T. Ferguson and S. Samuels, eds)}, 125: 37-58. 1991.

\bibitem {gil}J. Gilbert and F. Mosteller. \newblock Recognizing the maximum
of a sequence. \newblock {\em J. Am. Statist. Assoc.}, 61, 35-73, 1966.

\bibitem {ru}E. Lebensztayn, F. P. Machado and P. M. Rodriguez.
\newblock Limit Theorems for a general sthochastic rumour model. \newblock {\em arxiv.org/pdf/1003.4995}.

\bibitem {101}D. V. Lindley. \newblock Dynamic programming and decision
theory. \newblock {\em Journal
of the Royal Statistical Society. Series C (Applied Statistics)}, 10(1):39-51, 1961.

\bibitem {sonin}E.L. Presman and I.M. Sonin. \newblock The best choice problem
for a random number of objects. \newblock {\em Theory Prob. Applic.}, 17,
657-668. 1972.

\bibitem {2009}K.A. Szjowski. \newblock A rank-based selection with cardinal
payoffs and a cost of choice. \newblock {\em Sci. Math. Jpn.}, 69(2), 285-293. 2009.

\bibitem {posdoc}R.J. Vanderbei. \newblock The Optimal Choice of a Subset of a
Population. \newblock {\em Mathematics of Operations Research}, 5(4): 481-486. 1980.
\end{thebibliography}
\end{document}